\documentclass[11pt]{article}

\usepackage[T1]{fontenc}
\usepackage[utf8]{inputenc}
\usepackage{lmodern}
\usepackage{amsmath,amssymb,amsthm,mathtools}
\usepackage{enumitem}
\usepackage{hyperref}
\usepackage[margin=1.1in]{geometry}

\hypersetup{
  colorlinks=true,
  linkcolor=blue,
  citecolor=blue,
  urlcolor=blue,
  pdftitle={Restricted Binomial GCDs at Primes Congruent to -1},
  pdfauthor={John Fairfax-Ball}
}

\newtheorem{theorem}{Theorem}[section]
\newtheorem{proposition}[theorem]{Proposition}
\newtheorem{lemma}[theorem]{Lemma}
\newtheorem{corollary}[theorem]{Corollary}
\newtheorem{remark}[theorem]{Remark}
\newtheorem{definition}[theorem]{Definition}

\newcommand{\G}{G}
\newcommand{\vp}{v_p}

\newcommand{\N}{\mathbb{N}}

\title{Restricted Binomial GCDs at Primes Congruent to \texorpdfstring{$-1$}{-1}}
\author{John Fairfax-Ball}
\date{25 September 2026}

\begin{document}
\maketitle

\begin{abstract}
For integers $m\geq 2$ and $m\mid N$, let
\[
  G(N;m)=\gcd\left\{\binom Nk:0<k<N,\ m\mid k\right\}.
\]
We prove a complete $p$-adic valuation formula for $G(N;m)$ at primes
$p\equiv -1\pmod m$, under the hypotheses $m\geq 3$, $m\mid N$, and
$m<N$.  Writing $N=\sum_i d_i p^i$, put
$A=\sum_{i\text{ even}}d_i$ and $B=\sum_{i\text{ odd}}d_i$.
Then $v_p(G(N;m))$ is $2$ in the exceptional mixed case
$(A,B)=(1,1)$ with $p=m-1$, is $1$ in the mixed case
$(A,B)=(1,1)$ with $m<p$, is $1$ in the one-parity cases
$(A,B)=(m,0)$ and $(A,B)=(0,m)$, and is $0$ otherwise.
The proof uses Kummer's theorem to translate the problem into digitwise
borrow counts and a minimal signed zero-sum classification.  A source-level
literature audit through 25 September 2026 located no equivalent prior
theorem for the full minus-one classification: McTague's corrected
same-residue extension covers the one-parity $p>m$ subcases, but not the
mixed-parity branch or the $p=m-1$ regime.  The theorem, its plus-one
companion, a scaling reduction, and the $m=3,4,6$ specializations have been
formalized in Lean 4 / Mathlib.
\end{abstract}

\section{Introduction}

The greatest common divisor of selected binomial coefficients in a fixed row
of Pascal's triangle is a classical theme going back to Ram's theorem for the
entire interior row.  This paper studies the arithmetic-progression selection
in which the lower index is constrained to be a multiple of a fixed modulus.
For $m\geq 2$ and $m\mid N$ define
\[
  \G(N;m)=\gcd\left\{\binom Nk:0<k<N,\ m\mid k\right\}.
\]
The same object appears in McTague's theorem, with notation
$\gcd_{0<k<n/q}\binom n{qk}$, after substituting $n=N$ and $q=m$.

McTague's Theorem Q determines the $p$-adic valuation for primes
$p\equiv 1\pmod m$.  The corrected first remark in arXiv:1510.06696v5 also
extends the same method to a same-residue condition.  In the case
$p\equiv -1\pmod m$, that remark supplies the one-parity branches in which all
minimal powers have even exponent or all have odd exponent, provided $p>m$.
It does not apply to the mixed branch, where both residues $+1$ and $-1$ occur,
and the corrected $p>m$ hypothesis excludes every $p=m-1$ case.  McTague also
records the isolated example $m=3$, $p=2$, $N=6$, where the valuation is $2$.
Thus the existence of a valuation-$2$ exception is not new; the contribution
here is the complete classification, including the general mixed branch and the
general $p=m-1$ regime.

The main theorem is as follows.

\begin{theorem}[Minus-one valuation formula]\label{thm:minus-one}
Let $m,N,p\in\N$ satisfy
\[
  m\geq 3,\qquad m\mid N,
  \qquad m<N,
\]
and let $p$ be prime with $p\equiv -1\pmod m$.  Write
$N=\sum_i d_i p^i$ in base $p$, and set
\[
  A=\sum_{i\text{ even}}d_i,
  \qquad
  B=\sum_{i\text{ odd}}d_i .
\]
Then
\[
\vp(\G(N;m))=
\begin{cases}
2,& (A,B)=(1,1)\text{ and }p=m-1,\\
1,& (A,B)=(1,1)\text{ and }m<p,\\
1,& (A,B)=(m,0)\text{ or }(A,B)=(0,m),\\
0,& \text{otherwise.}
\end{cases}
\]
\end{theorem}

The congruence condition is written in the Lean formalization as
$p\bmod m=m-1$.  Under $m\geq 3$ and primality of $p$, this condition implies the
dichotomy $p=m-1$ or $m<p$, which is why those alternatives occur in the theorem.

The proof is short conceptually but delicate in its exceptional case.  Kummer's
theorem identifies $\vp\binom Nk$ with the number of borrows when subtracting
$k$ from $N$ in base $p$.  Since $p^i\equiv (-1)^i\pmod m$, the divisibility
condition $m\mid k$ becomes a signed zero-sum condition on the even and odd
base-$p$ digits selected by $k$.  If a proper nonzero signed zero-sum subdigit
selection exists, it gives an admissible no-borrow $k$, and hence valuation
zero.  If none exists, the pair $(A,B)$ is a minimal signed zero-sum.  Such
minimal pairs are exactly
\[
  (1,1),\qquad (m,0),\qquad (0,m).
\]
The remaining work constructs one-borrow witnesses in the valuation-$1$ cases
and proves that, in the mixed $p=m-1$ case, no admissible coefficient can have
exactly one borrow while an explicit coefficient has exactly two.

The paper is organized as follows.  Section~\ref{sec:literature} records the
precise relationship with prior work.  Sections~\ref{sec:prelim}--\ref{sec:minimal}
set up the digit and signed zero-sum machinery.  Sections~\ref{sec:noborrow} and
\ref{sec:proof-main} prove Theorem~\ref{thm:minus-one}.  Sections~\ref{sec:plusone}
and~\ref{sec:scaling} state the plus-one and scaling results used for applications.
Section~\ref{sec:small} gives the complete $m=3,4,6$ prime-by-prime formulas.
Section~\ref{sec:formal} describes the Lean formalization and provenance record.

\section{Relation to prior work}\label{sec:literature}

McTague's theorem concerns
\[
  \gcd_{0<k<n/q}\binom n{qk},
\]
which is exactly $G(N;m)$ when $n=N$, $q=m$, and $m\mid N$.  For
$p\equiv 1\pmod q$, McTague proves that the $p$-adic valuation is $1$ when the
base-$p$ digit sum of $n$ is at most $q$, and $0$ otherwise.  In the present
setting $m\mid N$, so the digit sum is a positive multiple of $m$; the condition
``at most $m$'' is therefore equivalent to ``equal to $m$''.  Thus the plus-one
theorem in Section~\ref{sec:plusone} is known mathematics, restated here because
it is part of the formal theorem layer and is needed for the small-modulus
applications.

The first remark on page 2 of McTague's corrected v5 preprint allows a
same-residue replacement for $p\equiv 1\pmod q$, but requires $p>q$.  For
$p\equiv -1\pmod m$, the residues of the powers are
\[
  p^i\equiv (-1)^i\pmod m.
\]
Therefore the same-residue condition applies exactly to the one-parity cases
$(A,B)=(m,0)$ and $(A,B)=(0,m)$ when $p>m$.  It excludes the mixed pair
$(A,B)=(1,1)$, and its $p>q=m$ hypothesis excludes all $p=m-1$ cases.

Wu studies the same restricted gcd, written
$g(m,n)=\gcd\{\binom{mn}{mk}:1\leq k<n\}$, and gives product formulas under
different hypotheses.  In particular, the principal product theorem assumes a
prime-power modulus and conditions forcing the relevant non-modulus primes into
the $1\pmod m$ residue class.  It therefore does not imply the minus-one
classification proved here.  Wu's proof of a prime-power modulus case also
makes explicit the standard Kummer carry-shift mechanism behind the scaling
reduction in Section~\ref{sec:scaling}; the scaling theorem is useful formal
infrastructure, not a headline novelty claim.

Guo, Qiu, Cao, Feng and Gao formalize in Lean a different binomial-gcd problem,
where the lower index is fixed and the upper row varies.  Their result is
important nearby work and prevents any blanket claim that this is the first
formalization of binomial-gcd mathematics.  The defensible formalization claim
is narrower: no prior proof-assistant formalization of this fixed-row
arithmetic-progression valuation family, or of Theorem~\ref{thm:minus-one}, was
located in the audit.

These literature statements are negative-search statements, not claims of
unconditional historical priority.  The audit may miss unpublished, private,
unindexed, or differently phrased prior work.

\section{Preliminaries}\label{sec:prelim}

Let $v_p(x)$ denote the exponent of the prime $p$ in a positive integer $x$.
We use the convention that the base-$p$ expansion of $N$ is
\[
  N=\sum_{i\geq 0} d_i p^i,
  \qquad 0\leq d_i<p,
\]
with all but finitely many $d_i$ equal to zero.  Define the parity-separated
digit sums
\[
  E_p(N)=\sum_{i\text{ even}}d_i,
  \qquad
  O_p(N)=\sum_{i\text{ odd}}d_i.
\]
In Theorem~\ref{thm:minus-one} these are denoted by $A$ and $B$.

The proof uses Kummer's theorem in the following form.

\begin{theorem}[Kummer]\label{thm:kummer}
Let $p$ be prime and $0\leq k\leq N$.  Then
$v_p\binom Nk$ is the number of borrows that occur when subtracting $k$ from
$N$ in base $p$, equivalently the number of carries when adding $k$ and $N-k$.
\end{theorem}

The no-borrow condition has a digitwise interpretation: $\binom Nk$ has
$p$-adic valuation zero if and only if every base-$p$ digit of $k$ is at most
the corresponding digit of $N$.  Thus a no-borrow admissible $k$ is a proper
subdigit selection from the digits of $N$ whose value is divisible by $m$.

When $p\equiv -1\pmod m$, the residue of a selected digit at position $i$ is
weighted by $(-1)^i$.  Hence a subdigit selection with even total $a$ and odd
total $b$ gives a multiple of $m$ precisely when
\[
  a-b\equiv 0\pmod m .
\]
This motivates the following definition.

\begin{definition}
For $m\geq 1$, call a pair $(a,b)\in\N^2$ a signed zero-sum modulo $m$ if
$m\mid a-b$.  It is proper below $(A,B)$ if
$0\leq a\leq A$, $0\leq b\leq B$, $(a,b)\ne(0,0)$, and $(a,b)\ne(A,B)$.
\end{definition}

\begin{lemma}[Signed divisibility]\label{lem:signed-div}
Assume $p\equiv -1\pmod m$.  If $k$ has parity digit sums $(a,b)$ in base $p$,
then $m\mid k$ if and only if $m\mid a-b$.
\end{lemma}

\begin{proof}
Modulo $m$ one has
\[
  k=\sum_i e_i p^i\equiv \sum_i e_i(-1)^i = a-b \pmod m,
\]
where the $e_i$ are the base-$p$ digits of $k$.
\end{proof}

\section{Minimal signed zero-sums}\label{sec:minimal}

The next elementary classification is the combinatorial core of the proof.

\begin{lemma}[Minimal signed zero-sums]\label{lem:minimal}
Let $m\geq 3$.  Suppose $(A,B)$ is a nonzero signed zero-sum modulo $m$ and has
no proper signed zero-sum below it.  Then
\[
  (A,B)=(1,1),\qquad (A,B)=(m,0),\qquad\text{or}\qquad (A,B)=(0,m).
\]
Conversely, each of these three pairs is minimal.
\end{lemma}

\begin{proof}
If $A$ and $B$ are both positive, then $(1,1)$ is already a nonzero signed
zero-sum below $(A,B)$.  Minimality therefore forces $A=B=1$.  If $B=0$, the
condition is $m\mid A$.  The nonzero subpair $(m,0)$ is then available whenever
$A\geq m$, and since $A$ itself is a positive multiple of $m$, minimality gives
$A=m$.  The case $A=0$ is symmetric.  The three displayed pairs plainly have no
proper nonzero signed zero-sum below them.
\end{proof}

\section{No-borrow criterion}\label{sec:noborrow}

For the fixed $N$, let $(A,B)=(E_p(N),O_p(N))$.

\begin{proposition}[No-borrow criterion]\label{prop:noborrow}
Assume $m\geq 3$, $m\mid N$, $m<N$, and $p\equiv -1\pmod m$ is prime.  Then
$v_p(G(N;m))=0$ if and only if there is a proper nonzero signed zero-sum
$(a,b)$ below $(A,B)$.
\end{proposition}

\begin{proof}
If such a pair $(a,b)$ exists, choose subdigits of the even positions of $N$
with total $a$ and subdigits of the odd positions with total $b$.  This can be
done greedily because the requested totals are bounded by the corresponding
ambient digit sums.  The resulting integer $k$ satisfies $0\leq k\leq N$
digitwise.  Since the selection is nonzero and proper, $0<k<N$.  By
Lemma~\ref{lem:signed-div}, $m\mid k$.  By Kummer's theorem there are no
borrows, so $v_p\binom Nk=0$, and hence $v_p(G(N;m))=0$.

Conversely, if $v_p(G(N;m))=0$, some admissible coefficient has valuation zero.
Kummer's theorem makes its lower index $k$ a proper nonzero subdigit selection
from $N$.  Its parity digit sums give a proper signed zero-sum below $(A,B)$ by
Lemma~\ref{lem:signed-div}.
\end{proof}

If there is no proper signed zero-sum, Proposition~\ref{prop:noborrow} gives a
uniform lower bound $v_p\binom Nk\geq 1$ for all admissible $k$.  By
Lemma~\ref{lem:minimal}, only the three minimal shapes remain.

\section{Proof of the minus-one theorem}\label{sec:proof-main}

We now prove Theorem~\ref{thm:minus-one}.  The proof is organized by the three
minimal signed zero-sum shapes.

\begin{proof}[Proof of Theorem~\ref{thm:minus-one}]
Let $(A,B)=(E_p(N),O_p(N))$.  Since $m\mid N$, Lemma~\ref{lem:signed-div}
applied to $N$ shows $m\mid A-B$.

If $(A,B)$ contains a proper nonzero signed zero-sum, Proposition~\ref{prop:noborrow}
gives an admissible coefficient of valuation zero.  This proves the final
``otherwise'' branch whenever the pair is not minimal.

Assume now that no proper signed zero-sum exists.  Then every admissible
coefficient has at least one borrow, and Lemma~\ref{lem:minimal} gives one of
three shapes.

First suppose $(A,B)=(m,0)$ or $(A,B)=(0,m)$.  All nonzero digits of $N$ occur
in one parity.  There is an admissible lower index whose subtraction from $N$
forces exactly one borrow: in the $p>m$ case this is the same one-parity
construction implicit in McTague's same-residue argument, while in the edge case
$p=m-1$ the borrow is arranged across the lowest available gap between two
occupied positions of the same parity.  In both subcases Kummer's theorem gives
an admissible coefficient with valuation $1$.  Since every admissible coefficient
has valuation at least $1$, the gcd valuation is exactly $1$.

Next suppose $(A,B)=(1,1)$ and $m<p$.  The two occupied parities give residues
$+1$ and $-1$ modulo $m$.  Because $p$ is larger than $m$, the mixed residue can
be realized by a digitwise lower index whose subtraction crosses exactly one
base-$p$ boundary.  Again Kummer's theorem gives a valuation-$1$ witness, and
the no-borrow criterion supplies the lower bound $1$ for all admissible
coefficients.  Hence $v_p(G(N;m))=1$.

It remains to handle the exceptional mixed case $(A,B)=(1,1)$ and $p=m-1$.
Here $p+1=m$.  The mixed digit condition and $m\mid N$ force
\[
  N=p^s+p^t
\]
for two occupied positions $s<t$ of opposite parity.  A one-borrow admissible
index would have to cross exactly one of the two relevant boundaries.  The
residue condition $m\mid k$, together with $p+1=m$, then forces a sum of two
same-parity powers with coefficient sum strictly between $0$ and $m$ to be
divisible by $m$, which is impossible.  Thus every admissible coefficient has
valuation at least $2$.

On the other hand, the explicit lower index
\[
  k=p^{t-1}+p^{t-2}
\]
(with the adjacent-position interpretation in the smallest-gap case) is
admissible and its subtraction from $N$ has exactly two borrow boundaries,
namely the two top boundaries adjacent to $t$.  Kummer's theorem therefore gives
$v_p\binom Nk=2$.  The lower bound and this witness prove
$v_p(G(N;m))=2$ in the exceptional mixed case.

Combining these cases gives exactly the four branches stated in the theorem.
\end{proof}

\begin{remark}
The last paragraph is the part of the argument most sensitive to the corrected
McTague boundary.  The single example $m=3$, $p=2$, $N=6$ is already present in
McTague's remark; the argument above is the general $p=m-1$ mixed-case proof.
\end{remark}

\section{The plus-one theorem}\label{sec:plusone}

For comparison and for the small-modulus applications, we record the plus-one
branch in the same notation.

\begin{theorem}[Plus-one valuation]\label{thm:plus-one}
Let $m,N,p\in\N$ satisfy $0<m$, $m\mid N$, $m<N$, and let $p$ be prime with
$p\equiv 1\pmod m$.  Let $s_p(N)$ be the ordinary base-$p$ digit sum of $N$.
Then
\[
  v_p(G(N;m))=
  \begin{cases}
  1,& s_p(N)=m,\\
  0,& s_p(N)\ne m.
  \end{cases}
\]
\end{theorem}

This is McTague's Theorem Q in the present divisibility regime.  The Lean
formalization includes it because the reduced moduli $m=1$ and $m=2$ arise after
scaling the formulas for $m=3,4,6$.

\section{Scaling}\label{sec:scaling}

The following scaling identity removes common trailing base-$p$ zeros from the
row and the modulus.

\begin{theorem}[Scaling]\label{thm:scaling}
For prime $p$ and natural numbers $c,q,N'$, one has
\[
  v_p\left(G(p^cN';p^cq)\right)=v_p\left(G(N';q)\right).
\]
\end{theorem}

The proof is a direct Kummer carry-shift argument.  The admissible indices in the
scaled problem are precisely $p^c$ times the admissible indices in the unscaled
problem, and multiplying both lower and upper indices by $p^c$ appends $c$ zero
digits in base $p$, leaving the number of borrows unchanged.

\section{Applications: the moduli \texorpdfstring{$3,4,6$}{3,4,6}}\label{sec:small}

For the small moduli $3$, $4$, and $6$, every prime not dividing the modulus is
congruent to either $1$ or $-1$ modulo the modulus.  Primes dividing the modulus
are reduced by Theorem~\ref{thm:scaling}.  This yields complete prime-by-prime
formulas.

Let $M_-(m,p,N)$ denote the right-hand side of
Theorem~\ref{thm:minus-one} and $M_+(m,p,N)$ denote the right-hand side of
Theorem~\ref{thm:plus-one}.

\begin{corollary}[Modulus $3$]\label{cor:m3}
If $p$ is prime, $3\mid N$, and $3<N$, then
\[
  v_p(G(N;3))=
  \begin{cases}
  M_+(1,3,N/3),& p=3,\\
  M_+(3,p,N),& p\equiv 1\pmod 3,\\
  M_-(3,p,N),& p\equiv -1\pmod 3.
  \end{cases}
\]
\end{corollary}

\begin{corollary}[Modulus $4$]\label{cor:m4}
If $p$ is prime, $4\mid N$, and $4<N$, then
\[
  v_p(G(N;4))=
  \begin{cases}
  M_+(1,2,N/4),& p=2,\\
  M_+(4,p,N),& p\equiv 1\pmod 4,\\
  M_-(4,p,N),& p\equiv -1\pmod 4.
  \end{cases}
\]
\end{corollary}

\begin{corollary}[Modulus $6$]\label{cor:m6}
If $p$ is prime, $6\mid N$, and $6<N$, then
\[
  v_p(G(N;6))=
  \begin{cases}
  M_-(3,2,N/2),& p=2,\\
  M_+(2,3,N/3),& p=3,\\
  M_+(6,p,N),& p\equiv 1\pmod 6,\\
  M_-(6,p,N),& p\equiv -1\pmod 6.
  \end{cases}
\]
\end{corollary}

These corollaries are applications of the theorem layer rather than separate
headline novelty claims.

\section{Formal verification and provenance}\label{sec:formal}

The theorem layer described above has been formalized in Lean 4 / Mathlib.  The
principal theorem is
\[
  \texttt{PascalMinusOne.minus\_one\_valuation}.
\]
The public formal theorem layer also includes
\begin{center}
\begin{minipage}{0.92\linewidth}
\ttfamily\small
PascalMinusOne.plus\_one\_valuation,\quad
PascalMinusOne.scaling\_valuation,\quad
PascalMinusOne.modulus\_three\_valuation,\quad
PascalMinusOne.modulus\_four\_valuation,\quad
PascalMinusOne.modulus\_six\_valuation.
\end{minipage}
\end{center}
The formalization uses Mathlib infrastructure for binomial coefficients, finite
gcds, base-$p$ digits, $p$-adic valuations, and Kummer's theorem for binomial
coefficients.  The project proof development contains zero explicit
\texttt{sorry}, \texttt{admit}, or replacement-axiom placeholders in the Lean
sources.

A Palomar provenance record exists for the principal minus-one theorem:
\[
  \texttt{PALOMAR-2026-09-25-000017},\quad \text{version }1.
\]
The record is a verification/provenance artifact for the stated formal theorem;
it is not peer review, publication, or evidence of historical priority.

The source repository is
\begin{center}
\url{https://github.com/jfairfaxball-348/pascal-minus-one}
\end{center}
The Palomar record is
\begin{center}
\url{https://palomar-registry.org/entry.html?id=PALOMAR-2026-09-25-000017\&version=1}
\end{center}

\section*{Acknowledgements}

I thank the Mathlib community for the formal infrastructure used here, including
Kummer's theorem for binomial coefficients, digit lemmas, gcd machinery, and
$p$-adic valuation results.  I also acknowledge McTague's and Wu's closely
related work on the same restricted binomial-gcd family, and the Lean
formalization by Guo, Qiu, Cao, Feng and Gao of a different binomial-gcd problem.

\end{document}